\documentclass{article}
\usepackage{latexsym,amssymb,amsmath,amsfonts,pictex,enumerate,amsthm,dsfont}
\usepackage{ mathrsfs }
\usepackage{tikz}
\usetikzlibrary{decorations.markings}
\usepackage{scalerel,stackengine}
\usepackage[active]{srcltx}

\theoremstyle{definition}
\newtheorem{theorem}{Theorem}
\newtheorem{definition}{Definition}
\newtheorem{lemma}[theorem]{Lemma}

\stackMath
\newcommand\reallywidehat[1]{
\savestack{\tmpbox}{\stretchto{
  \scaleto{
    \scalerel*[\widthof{\ensuremath{#1}}]{\kern-.6pt\bigwedge\kern-.6pt}
    {\rule[-\textheight/2]{1ex}{\textheight}}
  }{\textheight}
}{0.5ex}}
\stackon[1pt]{#1}{\tmpbox}
}
\title{Analytic structure in spaces of Lipschitz functions}
\author{Stephen Deterding\thanks{Department of Mathematics and Physics, Marshall University, Huntington WV 25705, USA}}
\date{}

\begin{document}

\maketitle

\begin{abstract}
    Let $U \subseteq \mathbb C$ be bounded and open. For $0 < \alpha < 1$, $A_\alpha(U)$ is the set of functions in the little Lipschitz class with exponent $\alpha$ that are analytic in a neighborhood of $U$. We consider three conditions, motivated by the properties of bounded point derivations, that show how the functions in $A_\alpha(U)$ can have additional analytic structure than would otherwise be expected. We prove an implication between conditions $(c)$ and $(b)$ and show that there is no implication between conditions $(a)$ and $(c)$.
\end{abstract}

\section{Introduction}

The focus of this paper is on spaces spanned by rational functions and when the functions in those spaces may have additional analytic structure. Let $X$ be a compact subset of the complex plane, let $R_0(X)$ denote the set of rational functions with poles off $X$ and let $R(X)$ denote the closure of $R_0(X)$ in the uniform norm. $R(X)$ is widely studied in connection with the approximation of analytic functions by rational functions; for example, Runge's theorem states that a function that is analytic in a neighborhood of $X$ can be uniformly approximated by rational functions with poles chosen outside $X$.

Some properties of $R(X)$ are well known. A function in $R(X)$ is continuous on $X$ and analytic on the interior of $X$; however, in general it will not be analytic or even differentiable on the boundary. Nevertheless, if the structure of $X$ is nice enough, it is possible for the functions in $R(X)$ to possess additional analytic structure on the boundary. One such example is the existence of a bounded point derivation at a boundary point. 

\begin{definition}
A \textbf{bounded point derivation} $D$ on $R(X)$ at $x$ is a nontrivial bounded linear functional on $R(X)$ such that $D(fg) = D(f) g(x) + D(g) f(x)$.
\end{definition}

\noindent Alternatively, there is a bounded point derivation on $R(X)$ at $x$ if there exists a constant $C >0$ such that

\begin{align*}
    |f'(x)| \leq C ||f||_X
\end{align*}

\noindent for all $f \in R_0(X)$, where $||\cdot||_X$ denotes the uniform norm on $X$. An important theorem of Wang \cite[Corollary 3.5]{Wang1973} shows that if there is a bounded point derivation on $R(X)$ at $x$, then every function in $R(X)$ has an approximate derivative at $X$. Thus, the existence of a bounded point derivation at a boundary point implies that $R(X)$ has a greater semblance of analyticity at the boundary than would otherwise be expected.

 Other ways that functions can be shown to have additional analytic structure include the existence of an analytic disk and the existence of a nontrivial Gleason part \cite{Cole, Ghosh}; however, these are not considered in the present paper. Motivated by the properties of bounded point derivations, Wang \cite{Wang1974} identified additional conditions that provide further ways that functions in $R(X)$ might have additional analytic structure at a boundary point. Before we state these conditions, we review the following definitions.

\begin{definition}
 A function $\phi(r)$ is called \textbf{admissible} if it is positive and non-decreasing on $(0, \infty)$ and if the associated function $\psi(r) = \frac{r}{\phi(r)}$ is also positive and non-decreasing on $(0, \infty)$ with $\psi(0^+) =0$.
 \end{definition}

 \noindent Examples of admissible functions are the power functions $\phi(r) = r^{\alpha}$ where $0<\alpha <1$ and the function $\phi(r) = \frac{r}{\log(1+r)}$.

\begin{definition}
 Suppose $f(z)$ has derivatives of order $t$ at $x$. The \textbf{error at $z$ of the $t$-th degree Taylor polynomial of $f$ about $x$} is denoted by $R_x^tf(z)$ and defined as

\begin{equation*}
    R_x^t f(z) = f(z) - \sum_{j=0}^t \dfrac{f^{(j)}(x)}{j!} (z-x)^j.
\end{equation*}
\end{definition}

\noindent Let $A_n(x)$ denote the annulus $\{z: 2^{-(n+1)} \leq \vert z-x\vert  \leq  2^{-n}\}$ and let $B(x,r)$ denote the disk centered at $x$ with radius $r$. Let $m$ denote $2$-dimensional Lebesgue measure. 

\begin{definition}
A set $E$ has \textbf{full area density at $x$} if $\dfrac{m(B(x,r) \setminus E)}{m(B(x,r))} \to 0$ as $r \to 0$.
\end{definition}

\begin{definition}
    A function $f(z)$ has an \textbf{approximate derivative} at $x$ if there exists a set $E$ with full area density at $x$ such that 

\begin{align*}
    \lim_{z\to x, z\in E} \frac{f(z)-f(x)}{z-x}
\end{align*}

\noindent is finite. The value of the approximate derivative is the value of this limit.
\end{definition}

\begin{definition}
    Let $X$ be a compact subset of $\mathbb C$. The \textbf{analytic capacity of $X$} is denoted by $\gamma(X)$ and defined as

\begin{align*}
    \gamma(X) = \sup |f'(\infty)|
\end{align*}

\noindent where the supremum is taken over all functions $f$ that satisfy the following conditions:

\begin{enumerate}
    \item $f$ is analytic on $\hat{\mathbb C} \setminus X$.
    \item $||f||_{\hat{\mathbb C} \setminus X} \leq 1$
    \item $f(\infty) = 0$.
\end{enumerate}
\end{definition} 

\noindent See \cite[pg. 196]{Gamelin} for additional information on analytic capacity.

Let $X$ be a compact subset of the plane with $x \in X$, let $t$ be a non-negative integer, and let $\phi(r)$ be an admissible function. Wang's three conditions are the following.

\begin{enumerate}[(A)]
    \item For each $\epsilon >0 $ the set 
    
    \begin{align*}
   & \{y\in X: \vert R_x^t f(y) \vert \leq \epsilon \phi(\vert y-x \vert)\vert y-x \vert^t \Vert f \Vert_X \\ &\text{ for all }  f \in R_0(X)\}
    \end{align*}
    
   \noindent has full area density at $x$.
    
    \item There exists a representing measure $\mu$ for $x$ on $R(X)$ such that $\mu(x) = 0$ and 
    
    \begin{equation*} 
    \int \dfrac{d \vert \mu \vert(z)}{\vert z-x \vert^t\phi(\vert z-x \vert)} < \infty.
    \end{equation*}
    
    \item The series 
    
    \begin{equation*}
        \sum_{n=1}^{\infty} \dfrac{2^{n(t+1)}\gamma(A_n(x) \setminus X)}{\phi(2^{-n})} 
    \end{equation*}
    
 \noindent converges.
    
\end{enumerate}

These conditions can be thought of as generalizations of important properties of bounded point derivations on $R(X)$. Condition (A) is a generalization of Wang's theorem \cite[Theorem 3.4]{Wang1973} that the existence of a bounded point derivation implies an approximate Taylor theorem for functions in $R(X)$, condition (B) generalizes the property that the existence of a bounded point derivation on $R(X)$ at $x$ implies the existence of a representing measure for $x$ on $R(X)$ \cite[pg. 351]{Wang1973}, and condition (C) generalizes Hallstrom's condition for the existence of bounded point derivations \cite{Hallstrom1969}. These three conditions are known to be equivalent when $\phi \equiv 1$, and conditions (B) and (C) are equivalent for all admissible $\phi$; however, if $\phi \not \equiv 1$ then condition (A) is not equivalent to the other two conditions. In this case (B) and (C) imply (A) but (A) does not imply (B) or (C) \cite{O'Farrell1978}.

We now consider the analogues of these conditions for related spaces of functions. A previous manuscript \cite{Deterding2024} considered the analogues of these conditions for the space $R^p(X)$, the closure of $R_0(X)$ in the $L^p$ norm, and showed that the corresponding conditions have the same relationships among themselves as their analogues for $R(X)$. In the present manuscript, we will focus on spaces related to Lipschitz approximation. Much study has been done in this area \cite{Lord1994, O'Farrell2016, Sherbert}, but there are still many unanswered questions and open problems. The first goal of this paper is to determine the correct analogues of Wang's conditions for the case of analytic Lipschitz functions, which will be done in the next section. After that we will verify two implications which we are able to show; that $(c)$ implies $(b)$ in Section 3 and that $(a)$ does not imply $(c)$ in Section 4.

\section{Lipschitz functions and the Lipschitz analogues of Wang's Conditions}

To determine the correct analogues of Wang's conditions for Lipschitz functions, we have to consider the functional properties of Lipschitz spaces, which are very different from those of $R(X)$ and $R^p(X)$. We start by defining a Lipschitz condition.

\begin{definition}
Let $0< \alpha < 1$. A function $f$ is said to be a \textbf{Lipschitz function with exponent $\alpha$} if there exists a constant $k>0$ such that for all $z,w \in \mathbb C$

\begin{align}
\label{lipschitz}
    |f(z)-f(w)| \leq k |z-w|^{\alpha}. 
\end{align}
\end{definition}

\noindent For such a function, the Lipschitz seminorm, denoted as $\Vert \cdot\Vert _{\text{Lip}_{\alpha}}'$ is defined to be the smallest value of $k$ that satisfies \eqref{lipschitz}. It is only a seminorm because $k=0$ whenever $f$ is a constant function. The Lipschitz norm with exponent $\alpha$, denoted as $\Vert \cdot\Vert _{\text{Lip}_{\alpha}}$, is defined by

\begin{align*}
    \Vert f\Vert _{\text{Lip}_{\alpha}} = \Vert f\Vert _{\text{Lip}_{\alpha}}' +  ||f(z)||_{\infty}.
\end{align*}

\noindent For an arbitrary set $X$, the Lipschitz seminorm with exponent $\alpha$ on $X$, denoted as $\Vert \cdot\Vert _{\text{Lip}_{\alpha}(X)}'$ is defined to be the smallest value of $k$ that satisfies \eqref{lipschitz} for all $z,w \in X$. The Lipschitz norm with exponent $\alpha$ on $X$, denoted as $\Vert \cdot\Vert _{\text{Lip}_{\alpha}(X)}$, is defined by

\begin{align*}
    \Vert f\Vert _{\text{Lip}_{\alpha}(X)} = \Vert f\Vert _{\text{Lip}_{\alpha}(X)}' +  ||f(z)||_{X}.
\end{align*}

The space of Lipschitz functions on $\mathbb C$ equipped with the Lipschitz norm forms a Banach space, which we denote by Lip$_\alpha(\mathbb C)$. This set contains an important subspace, the little Lipschitz class, which plays an important role in the theory of Lipschitz functions.

\begin{definition}
The \textbf{little Lipschitz class with exponent $\alpha$}, which is denoted by lip$_\alpha(\mathbb C)$, is the set of functions in Lip$_\alpha(\mathbb C)$ that also satisfy the property that 
\begin{equation*}
    \lim_{\delta \to 0^+} \sup_{0<|z-w|<\delta} \dfrac{|f(z)-f(w)|}{|z-w|^{\alpha}} = 0.
\end{equation*}
\end{definition}

\noindent Let $U$ be a bounded open subset of $\mathbb C$. The focus of this paper is on the following space of analytic functions on $U$ that satisfy a global Lipschitz condition. 

\begin{definition}
 $A_{\alpha}(U)$ denotes the set of functions in the little Lipschitz class with exponent $\alpha$ that are also analytic on $U$. 
\end{definition}

A key difference between the Lipschitz spaces and $R(X)$ is in how their linear functionals can be represented by integration. It follows from the work of De Leeuw \cite{deLeeuw1961} that linear functionals acting on $A_{\alpha}(U)$ have the following integral representation. If $L$ is a linear functional acting on $A_{\alpha}(U)$ and $Y$ is the closure of $U$, then there exists a Borel-regular measure $\mu$ on $Y \times Y$ such that

\begin{align*}
    L(f) = \int_{Y \times Y} \frac{f(z)-f(w)}{|z-w|^{\alpha}} d\mu(z,w)
\end{align*}

\noindent whenever $f \in$ $A_{\alpha}(U)$. In general $L$ cannot be represented by direct integration against a measure in contrast to linear functionals on $R(X)$.

Another important difference between $A_{\alpha}(U)$ and $R(X)$ is found in their associated capacities. Every Banach space has an associated capacity, and many properties of that space, such as the existence of bounded point derivations, are characterized using this capacity. For $R(X)$ the associated capacity is analytic capacity, while for $A_{\alpha}(U)$ the associated capacity is lower $(1+\alpha)$-dimensional Hausdorff content. This can be seen in condition (C) and \cite{Hallstrom1969} for $R(X)$ and \cite{Lord1994} for $A_\alpha(U)$. The lower $(1+\alpha)$-dimensional Hausdorff content is defined as follows. Let $E \subseteq \mathbb C$.

\begin{definition}
A \textbf{measure function} is a non-decreasing function $h(t)$, $t\geq 0$, such that $h(t) \to 0$ as $t \to 0$.
\end{definition}

\begin{definition}
    The \textbf{Hausdorff content} associated to a measure function $h$ is denoted by $M^h(E)$ and defined by 

    \begin{equation*}
    M^h(E) = \inf \sum_j h(r_j),
\end{equation*}

\noindent where the infimum is taken over all countable covers of $E$ by squares with sides of length $r_j$.
\end{definition}

\begin{definition}
 The \textbf{lower $(1+\alpha)$-dimensional Hausdorff content of $E$} is denoted by $M_*^{1+\alpha}(E)$ and defined by 

\begin{equation*}
  M_*^{1+\alpha}(E)= \sup M^h(E) 
\end{equation*}

\noindent where the supremum is taken over all measure functions $h$ such that $h(t) \leq t^{1+\alpha}$ and $h(t)t^{-(1+\alpha)} \to 0$ as $t\to 0^+$.
\end{definition}

\noindent Furthermore, up to a constant multiplicative bound, the infimum can be taken over countable coverings of $E$ by dyadic squares \cite[pg. 61, Lemma 1.4]{Garnett}. It follows from the definition that lower $(1+\alpha)$-dimensional Hausdorff content is monotone; that is, $U \subseteq V$ implies $M_*^{1+\alpha}(U) \leq M_*^{1+\alpha}(V)$. Another key property of Hausdorff content is that for a disk $B$ with radius $r$, $M_*^{1+\alpha}(B) = r^{1+\alpha}$.

With these distinctions between $R(X)$ and $A_{\alpha}(U)$ understood, we can now determine the analogue of Wang's conditions for $A_\alpha(U)$. Let $U$ be a bounded open subset of $\mathbb C$ with $x \in U$, let $Y$ denote the closure of $U$, let $t$ be a non-negative integer, and $\phi$ be an admissible function. The analogues of Wang's conditions for $A_\alpha(U)$ are:

\begin{enumerate}[(a)]
    \item For each $\epsilon >0 $ the set 
    
    \begin{align*}
   & \{y\in U: \vert R_x^t f(y) \vert \leq \epsilon \phi(\vert y-x \vert)\vert y-x \vert^t \Vert f \Vert_{\text{Lip}_\alpha(\mathbb C)} \\ &\text{ for all }  f \in A_{\alpha}(U)\}
    \end{align*}
    
 \noindent has full area density at $x$.
    
    \item There exists a Borel regular measure $\mu$ on $Y \times Y$ such that for all functions $f \in A_{\alpha}(U)$

    \begin{align*}
        f(x) = \int \frac{f(z)-f(w)}{|z-w|^{\alpha}} d\mu(z,w) 
    \end{align*}
    
    \noindent and 
    
    \begin{equation*} 
    \int \dfrac{d |\mu|(z,w)}{\vert z-x \vert^t\phi(\vert z-x \vert)} < \infty.
    \end{equation*}
    
    \item The series 
    
    \begin{equation*}
        \sum_{n=1}^{\infty} \dfrac{2^{n(t+1)} M_*^{1+\alpha}(A_n(x) \setminus U)}{\phi(2^{-n})} 
    \end{equation*}
    
 \noindent converges.
    
\end{enumerate}

\noindent Based on the relationships of the corresponding conditions for $R(X)$ and $R^p(X)$ we conjecture that $(b)$ and $(c)$ are equivalent and imply $(a)$ but $(a)$ does not imply $(b)$ or $(c)$.

We close this section with the following Lipschitz analog of Melnikov's theorem \cite[Theorem 4]{Melnikov} for $R(X)$, which will be used several times in the following sections. See \cite[2.6]{Lord1994} for a proof. 

\begin{lemma}
\label{Lord}
     Let $\Gamma$ be a piecewise analytic curve that encloses a region $U$ and is free of outward pointing cusps, and let $0<\alpha<1$. Then there is a constant $C$ which only depends on the curve $\Gamma$ such that

\begin{equation*} 
\left| \int_{\Gamma} f(z) dz\right| \leq C M_*^{1+\alpha}(K) ||f||_{\text{Lip}_\alpha(K)}'.
\end{equation*}

\noindent whenever $f \in $ lip$_\alpha(\mathbb C)$
 is analytic on $U \setminus K$, where $K$ is a compact subset of $U$.

\end{lemma}

\section{(c) implies (b)}

We first prove that condition $(c)$ implies condition $(b)$. By letting $\phi(r)$ be an arbitrary positive non-decreasing function we may assume that $t=0$ without loss of generality.

\begin{theorem}
\label{c then b}
 Let $\phi(r)$ be a positive non-decreasing function, let $U$ be an open subset of the complex plane with $x \in U$, and let $Y$ denote the closure of $U$. Suppose

\begin{align*}
    \sum_{n=1}^{\infty} 2^n \phi(2^{-n}) M^{1+\alpha}_*(A_n(x) \setminus U) < \infty.
\end{align*}

\noindent Then there exists a Borel regular measure $\mu$ on $Y \times Y$ such that for all functions $f \in A_\alpha(U)$,

\begin{align*}
    f(x) = \int \frac{f(z)-f(w)}{|z-w|^{\alpha}} d\mu(z,w)
\end{align*}

\noindent and 

\begin{align*}
    \int \frac{d\mu(z,w)}{\phi(|z-x|)} < \infty.
\end{align*}
    
\end{theorem}

\begin{proof}
    Without loss of generality, we may assume that $x=0$. Let $A_n = A_n(0)$ and let $T_n$ be a linear functional on $A_\alpha(U \cap A_n)$ defined by

\begin{align*}
    T_n(f) = -\int_{A_n} \frac{f(z)}{z} dz
\end{align*}

\noindent and let $f \in A_\alpha(U)$. It follows from Lemma \ref{Lord} that 

\begin{align*}
    \left| \int_{A_n} \frac{f(z)}{z} dz \right| \leq C 2^n M_*^{1+\alpha}(A_n \setminus U) ||f||_{\text{Lip}_{\alpha}(A_n \setminus U)}'
\end{align*}

\noindent and thus $T_n$ is a bounded linear functional on $A_{\alpha}(U \cap A_n)$. Let $Y_n$ denote the closure of $U \cap A_n$. It follows from De Leeuw's representation of lip$\alpha^*$ that there is a Borel-regular measure $\nu_n$ on $Y_n \times Y_n$ such that

\begin{align*}
    T_n(f) = \int \frac{f(z)-f(w)}{|z-w|^{\alpha}} d\nu_n(z,w)
\end{align*}

\noindent and 

\begin{align*}
    \int |\nu_n(z,w)| \leq C 2^n M_*^{1+\alpha}(A_n \setminus U).
\end{align*}

\noindent Let $\displaystyle \nu = \frac{1}{2\pi i} \sum_{n=1}^{\infty} \nu_n$ and let $T$ be the linear functional on $A_\alpha(U \cap A_0)$ defined by 

\begin{align*}
    T(f) = \frac{1}{2\pi i}\int_{|z|=\frac{1}{2}} \frac{f(z)}{z} dz.
\end{align*}

\noindent Then $|T(f)| \leq 2||f||_{\text{Lip}_{\alpha}}(U \cap A_0)$ and thus $T$ is bounded. Hence, there is a Borel regular measure $\lambda$ on $Y_0 \times Y_0$ such that 

\begin{align*}
    T(f) = \int \frac{f(z)-f(w)}{|z-w|^{\alpha}} d\lambda(z,w)
\end{align*}

\noindent and 

\begin{align*}
    \int |\lambda(z,w)| \leq C.
\end{align*}

 Since $f$ is analytic in a neighborhood of $0$, there exists a positive integer $N$ such that $f$ is analytic on the disk $\{z:|z| \leq 2^{-N}\}$. Hence, by the Cauchy integral formula, there exists a constant $N>0$ such that $\displaystyle f(0) = \frac{1}{2\pi i}\int_{|z|=2^{-N}} \frac{f(z)}{z}dz$. Thus,

\begin{align*}
f(0) &= \frac{1}{2\pi i}\int_{|z|=2^{-N}} \frac{f(z)}{z}dz\\
 &= \frac{1}{2\pi i}  \int_{|z|=\frac{1}{2}} \frac{f(z)}{z} dz - \frac{1}{2\pi i} \sum_{n=1}^{\infty} \int_{\partial A_n} \frac{f(z)}{z} dz\\
    &= \int \frac{f(z)-f(w)}{|z-w|^{\alpha}} d\lambda(z,w) + \int \frac{f(z)-f(w)}{|z-w|^{\alpha}} \frac{1}{2\pi i} \sum_{n=1}^{\infty} d\nu_n(z,w)\\
     &= \int \frac{f(z)-f(w)}{|z-w|^{\alpha}} d\lambda(z,w) + \int \frac{f(z)-f(w)}{|z-w|^{\alpha}} d\nu.
\end{align*}

\noindent Let $\mu = \lambda + \nu$. Then 

\begin{align*}
    f(0) = \int \frac{f(z) - f(w)}{|z-w|^{\alpha}} d \mu
\end{align*}

\noindent as desired. Lastly, since $\nu_n$ has support on $Y_n \times Y_n$ it follows that

\begin{align*}
    \int \frac{|d\mu(z,w)|}{\phi(|z|)} &\leq \frac{1}{2\pi} \int \phi(|z|)^{-1} \left| \sum_{n=1}^{\infty} \nu_n(z,w) \right| + \int \phi(|z|)^{-1} |\lambda|\\
    &\leq\frac{1}{2\pi} \int \sum_{n=1}^{\infty} \phi(|2^{-n}|)^{-1} |\nu_n(z,w)| + C\\
    &\leq C \sum_{n=1}^{\infty}  \phi(|2^{-n}|)^{-1} 2^n M_*^{1+\alpha}(A_n \setminus U) < \infty,
\end{align*}

\noindent thus completing the proof.
    
\end{proof}

\section{(a) does not imply (c)}

We now present an example showing that, in general, condition $(a)$ does not imply condition $(c)$. For this example, we consider the case $t=1$ because this construction cannot be used to give a counterexample when $t=0$ as demonstrated in the appendix. We do not know if $(a)$ implies $(c)$ when $t=0$.

\begin{theorem}
\label{main}
Let $\phi(r)$ be an admissible function such that $\phi(0^{+})=0$, and let $0<\alpha<1$. Then there is an open set $U$ containing $0$ such that, for each $\epsilon>0$ the set $\{y\in U: \vert f(y)-f(0)-f'(0)y\vert \leq \epsilon \phi(\vert y\vert ) \vert y \vert \Vert f\Vert _{\text{Lip}_\alpha} \text{ for all } f \in A_{\alpha}(U)\}$ has full area density at $0$, but

\begin{equation*}
    \sum_{n=1}^{\infty} \phi(2^{-n})^{-1} 4^{n} M_*^{1+\alpha}(A_n(0) \setminus U) = \infty.
\end{equation*}

\end{theorem}

We first describe the construction of the set $U$. The main idea is to start with the unit disk and remove a single disk from each annulus $A_n(0) = \{2^{-(n+1)} \leq |z| \leq 2^{-n}\}$. (See Figure \ref{roadrunner}.)  Let $a_n = \frac{3}{4}\cdot 2^{-n}$ be a sequence of points on the positive real axis. These points are chosen so that for each $n$ the point $a_n$ lies in the annulus $A_n(0) = \{2^{-(n+1)} \leq |z| \leq 2^{-n}\}$. We can find a subsequence which we will still denote as $a_n$ such that $\phi(a_n) < \frac{1}{2} \phi(a_{n-1})$ and 

\begin{align*}
    \sum_{n=1}^{j-1} \psi(a_n)^{-1} < \psi(a_j)^{-1},
\end{align*}

\noindent where $\psi(r) = r\phi(r)^{-1}$ is the associated function to $\phi$. Let $r_n$ be a sequence of radii such that $r_n^{1+\alpha} = a_n^2 n^{-1} \phi(a_n)$. Let $B(a,r)$ denote the open disk centered at $a$ with radius $r$, let $\Delta$ denote the open unit disk in $\mathbb{C}$ and let $\displaystyle U = \Delta \setminus \bigcup_{n=1}^{\infty} \overline{B(a_n,r_n)}$. We will prove Theorem \ref{main} by verifying two lemmas. First, we show that condition $(c)$ does not hold for this set.

\begin{figure}
    \centering
    \includegraphics[width=0.3\linewidth]{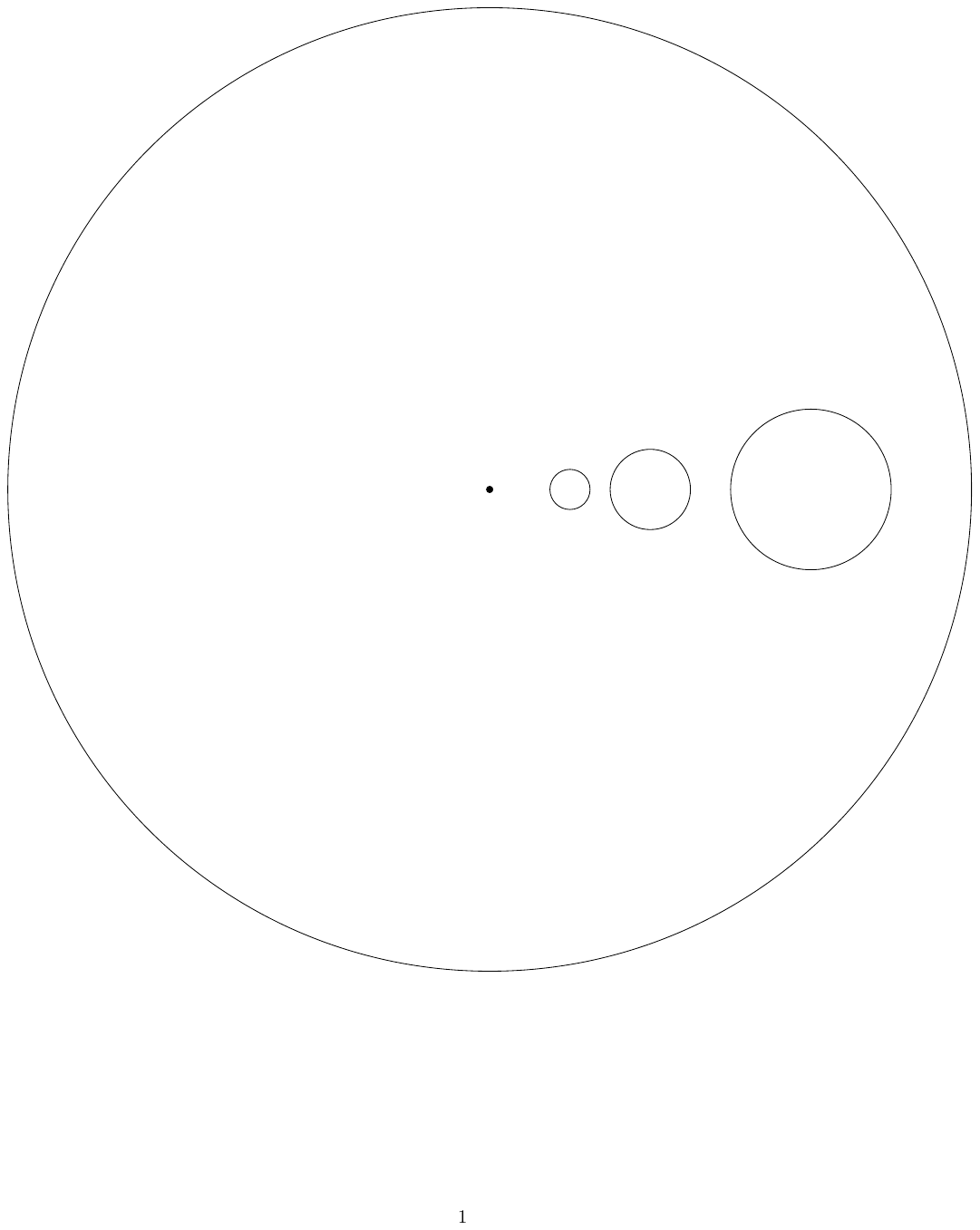}
    \caption{The set $U$ is constructed by removing a disk from each annulus $A_n(0) = \{2^{-(n+1)} \leq |z| \leq 2^{-n}\}$ in the unit disk.}
    \label{roadrunner}
\end{figure}

\begin{lemma}
\label{c is false}
    The set $U$ described above has the property that 

\begin{equation*}
    \sum_{n=1}^{\infty} \phi(2^{-n})^{-1} 4^{n} M_*^{1+\alpha}(A_n(0) \setminus U) = \infty.
\end{equation*}
    
\end{lemma}

\begin{proof}
  Let $A_n = A_n(0)$. Since $A_n \setminus U = B(a_n, r_n)$ and for a disk $B$ with radius $r$, $M_*^{1+\alpha}(B) = r^{1+\alpha}$, it follows that

\begin{align*}
   \sum_{n=1}^{\infty} \phi(2^{-n})^{-1} 4^{n} M_*^{1+\alpha}(A_n \setminus U) &= \sum_{n=1}^{\infty} \phi(2^{-n})^{-1} 4^{n} a_n^2 n^{-1} \phi(a_n)\\
   &= \frac{9}{16}\sum_{n=1}^{\infty} \phi(2^{-n})^{-1} n^{-1} \phi\left(\frac{3}{4}2^{-n}\right).
\end{align*}

\noindent Since $\phi(r) = \frac{r}{\psi(r)}$ and $\psi(r)$ is non-decreasing,

\begin{align*}
    \phi(2^{-n})^{-1} \phi\left(\frac{3}{4}2^{-n}\right) = \frac{3}{4} \psi \left(\frac{3}{4} 2^{-n}\right)^{-1} \psi(2^{-n}) \geq \frac{3}{4}.
\end{align*}

\noindent Hence,

\begin{align*}
    \sum_{n=1}^{\infty} \phi(2^{-n})^{-1} 4^{n} M_*^{1+\alpha}(A_n \setminus U) \geq \frac{27}{64} \sum_{n=1}^{\infty} n^{-1} = \infty,
\end{align*}

\noindent and the lemma is proved.

\end{proof}

The next lemma shows that $U$ satisfies condition (a).

\begin{lemma}
\label{a is true}
    Let $U$ be the set constructed in the paragraph before Lemma \ref{c is false}. For each $\epsilon>0$ the set $\{y\in U: \vert f(y)-f(0)-f'(0)y\vert \leq \epsilon \phi(\vert y\vert ) \vert y \vert \Vert f\Vert _{\text{Lip}_\alpha} \text{ for all } f \in A_{\alpha}(U)\}$ has full area density at $0$.
\end{lemma}

\begin{figure}[h!]
    \centering
    \includegraphics[width=\linewidth]{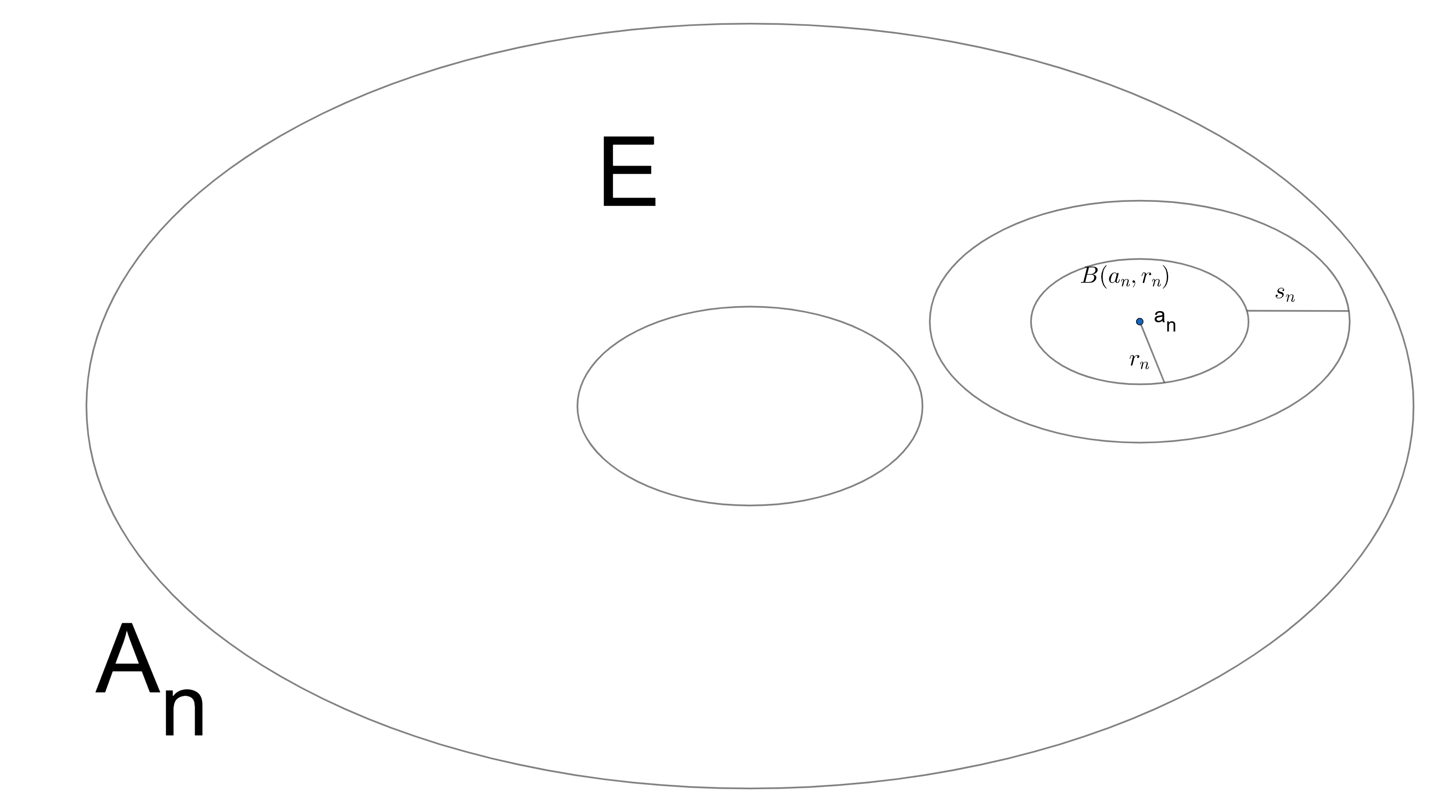}
    \caption{The portion of $E$ in the annulus $A_n$ consists of the compliment of the ball $B(a_n, r_n + s_n)$ in $A_n$.}
    \label{exceptional set}
\end{figure}

\begin{proof}

   Let $d_n(z)$ denote the distance from $z$ to $B(a_n,r_n)$, let $s_n = \frac{1}{7} n^{-1/3} a_n$, and let $E = \displaystyle \bigcup_{n=1}^{\infty} \{z \in A_n: d_n(z) \geq s_n\}$. (See Figure \ref{exceptional set}.) We claim that $E$ has full area density at $0$. To show this, we first note that since $r_n^{1+\alpha} = a_n^2 n^{-1} \phi(a_n)$,

   \begin{align*}
       \frac{r_n^{1+\alpha}}{s_n^{1+\alpha}} = 7^{1+\alpha} a_n^{1-\alpha}n^{\frac{\alpha-2}{3}}\phi(a_n) \to 0
   \end{align*}

   \noindent as $n\to \infty$ and thus $r_n < s_n$ for $n$ sufficiently large. Then, since $B(0,2^{-j}) \setminus E$ is the union of all disks $B(a_n, r_n +s_n)$, where the union is taken over all $n\geq j$ and $s_{n+1} < \frac{1}{2} s_n$, it follows that

\begin{align*}
    \frac{m(B(0,2^{-j})\setminus E)}{m(B(0,2^{-j}))} &= \frac{\pi\sum_{n=j}^{\infty} (r_n+s_n)^2}{\pi (2^{-j})^2}\\
    &<\frac{9\sum_{n=j}^{\infty} s_n^2}{4a_j^2}\\
    &<\frac{9\left(\sum_{n=0}^{\infty} 2^{-2n} \right) s_j^2}{4a_j^2}\\
    &= \frac{3}{49} j^{-2/3}.
\end{align*}

   \noindent Thus, $\displaystyle \frac{m(B(0,2^{-j})\setminus E)}{m(B(0,2^{-j}))} \to 0$ as $j \to \infty$ and $E$ has full area density at $0$.

   \bigskip

    Let $\epsilon>0$. We will show that if $y \in E$ then $|f(y)-f(0)-f'(0)y| < \epsilon \phi(|y|) |y| \Vert f \Vert_{\text{Lip}_\alpha}$ for all $f \in A_{\alpha}(U)$. Note that if we let $g(z) = f(z) -f(0) -f'(0)z$, then $g(0) = 0$, $g'(0) = 0$, and $g(y) = f(y) -f(0) - f'(0)y$. Hence, we may suppose that $f(0) = 0$ and $f'(0) = 0$. If $y \in E \cap A_N$, then it follows from the Cauchy integral formula that 

\begin{align*}
    f(y)-f(0) -f'(0)y &= \frac{1}{2 \pi i} \sum_{n=1}^{\infty} \left(\int_{\partial B(a_n,r_n)} \frac{f(z)}{z-y} dz- \int_{\partial B(a_n,r_n)} \frac{f(z)}{z}dz -y\int_{\partial B(a_n,r_n)} \frac{f(z)}{z^2}dz\right) \\
    &= \frac{y^2}{2 \pi i} \sum_{n=1}^{\infty} \int_{\partial B(a_n,r_n)} \frac{f(z)}{z^2(z-y)}.
\end{align*}

\noindent It follows from Lemma \ref{Lord} that

\begin{align*}
    |f(y)-f(0) -f'(0)y| \leq \frac{|y|^2}{2\pi} C \sum_{n=1}^{\infty} M_*^{1+\alpha} (A_n \setminus U) \left\Vert \frac{f(z)}{z^2(z-y)} \right\Vert_{\text{Lip}_{\alpha}(B(a_n,r_n))}'.
\end{align*}

\noindent  Using the triangle inequality, we obtain the following upper bound for the Lipschitz seminorm on $B(a_n, r_n)$.

\begin{align*}
    \left\Vert \frac{f(z)}{z^2(z-y)} \right\Vert_{\text{Lip}_{\alpha}(B(a_n,r_n))}' &= \sup_{z \neq w; z,w \in B(a_n,r_n)} \frac{\left| \frac{f(z)}{z^2(z-y)} - \frac{f(w)}{w^2(w-y)} \right|}{|z-w|^{\alpha}} \\
    &\leq \sup_{z \neq w; z,w \in B(a_n,r_n)} \frac{|f(z)-f(w)|}{|z-w|^{\alpha} |z|^2 |z-y|} + \sup_{z \neq w; z,w \in B(a_n,r_n)} \frac{|f(w)| |w^2(w-y) -z^2(z-y)|}{|z-w|^{\alpha} |z|^2 |z-y| |w|^2 |w-y|}\\
    & = (I) + (II)
\end{align*}

\noindent It follows directly from the definition of the Lipschitz seminorm that 

\begin{align*}
(I) \leq \left( \sup_{z \in B(a_n,r_n)}|z|^{-2} |z-y|^{-1}\right) ||f||_{\text{Lip}_{\alpha}}'.
\end{align*}

\noindent To get a bound for $(II)$, we first note that since $f(0) = 0$, 

\begin{align*}
   (II)  \leq \sup_{z \neq w; z,w \in B(a_n,r_n)} \left( \frac{ |w^2(w-y) -z^2(z-y)|}{|z-w|^{\alpha} |z|^2 |z-y| |w|^{2-\alpha} |w-y|} \right) ||f||_{\text{Lip}_{\alpha}}'.
\end{align*}

\noindent Next, a calculation shows that $w^2(w-y) -z^2(z-y) = (w-z)[w^2+w(z-y)+z(z-y)]$. Hence,

\begin{align*}
    (II) \leq \sup_{z \neq w; z,w \in B(a_n,r_n)} \left( \frac{|w-z|^{1-\alpha}}{|z|^2|z-y||w|^{-\alpha}|w-y|} + \frac{|w-z|^{1-\alpha}}{|z|^2|w|^{1-\alpha}|w-y|} + \frac{|w-z|^{1-\alpha}}{|z||w|^{2-\alpha}|w-y|} \right) ||f||_{\text{Lip}_{\alpha} }'.
\end{align*}

\noindent Now, since $z,w \in B(a_n,r_n)$, $|w-z| \leq |w|$, $\frac{1}{2} a_n \leq |z| \leq 2 a_n$, and $\frac{1}{2} a_n \leq |w| \leq 2 a_n$. Thus, 

\begin{align*}
    (I) \leq 4 a_n^{-2} d_n(y)^{-1} ||f||_{\text{Lip}_{\alpha} }'
\end{align*}

\noindent and

\begin{align*}
    (II) \leq \left( 4 a_n^{-1} d_n(y)^{-2} + 4 a_n^{-2} d_n(y)^{-1} + 4 a_n^{-2} d_n(y)^{-1}\right) ||f||_{\text{Lip}_{\alpha} }',
\end{align*}

\noindent where $d_n(y)$ denotes the distance from $y$ to $B(a_n, r_n)$. Hence,

\begin{align*}
    \left\Vert \frac{f(z)}{z^2(z-y)} \right\Vert_{\text{Lip}_{\alpha}(B(a_n,r_n))}' \leq C \left( a_n^{-2} d_n(y)^{-1} + a_n^{-1} d_n(y)^{-2} \right) ||f||_{\text{Lip}_{\alpha} }'
\end{align*}

\noindent and

\begin{align*}
    |f(y)-f(0) -f'(0)y| \leq C |y|^2 ||f||_{\text{Lip}_{\alpha} }' \sum_{n=1}^{\infty} \left( a_n^{-2} d_n(y)^{-1} + a_n^{-1} d_n(y)^{-2} \right) M_*^{1+\alpha}(A_n \setminus U).
\end{align*}

\noindent Since $A_n \setminus U = B(a_n, r_n)$,

\begin{align*}
    |f(y)-f(0) -f'(0)y| &\leq C ||f||_{\text{Lip}_{\alpha} }'\left( |y|^2\sum_{n=1}^{\infty}  a_n^{-2} d_n(y)^{-1} r_n^{1+ \alpha} + |y|^2\sum_{n=1}^{\infty}  a_n^{-1} d_n(y)^{-2}  r_n^{1+\alpha} \right)\\
    &= C ||f||_{\text{Lip}_{\alpha} }' \left( (A) + (B) \right).
\end{align*}

To obtain bounds for (A) and (B), we first obtain an estimate for $d_n(y)$. If $y \in E \cap A_N$, then $d_N(y) \geq s_N$. However, if $n < N$, then 

\begin{align*}
    d_n(y) \geq (a_n -r_n -s_n) - 2^{-N} \geq \frac{5}{7} a_n - 2^{-N} = \frac{15}{28} 2^{-n} -2^{-N} \geq \frac{1}{28} 2^{-n} = \frac{1}{21} a_n.
\end{align*}

\noindent On the other hand, if $n > N$, then

\begin{align*}
    d_n(y) \geq 2^{-(N+1)} - (a_n + r_n + s_n) \geq 2^{-(N+1)} - \frac{9}{7} a_n \geq 2^{-(N+1)} - \frac{27}{28} 2^{-n} \geq \frac{1}{28} 2^{-(N+1)} = \frac{1}{42} a_N.
\end{align*}

Now we derive an upper bound for (A). Since $y \in A_N$, 

\begin{align*}
    (A) &= |y|^2\left(\sum_{n=1}^{N-1}  a_n^{-2} d_n(y)^{-1} r_n^{1+\alpha} + a_N^{-2} d_N(y)^{-1} r_N^{1+\alpha} + \sum_{n=N+1}^{\infty}  a_n^{-2} d_n(y)^{-1} r_n^{1+\alpha}\right)\\
    &\leq C|y|^2 \left( \sum_{n=1}^{N-1}  a_n^{-3}  r_n^{1+\alpha} + a_N^{-2} s_N^{-1} r_N^{1+\alpha} + \sum_{n=N+1}^{\infty}  a_n^{-2} a_N^{-1} r_n^{1+\alpha} \right).
\end{align*}

\noindent Observe that 

\begin{align*}
    \frac{|y|}{\phi(|y|)} \leq \psi\left(\frac{4}{3}a_N\right) \leq \frac{\frac{4}{3}a_N}{\phi(a_N)} = \frac{4}{3} \psi(a_N).
\end{align*}

\noindent Hence $|y| \leq \frac{4}{3} \phi(|y|) \psi(a_N)$ and

\begin{align*}
    (A) &\leq C|y|\phi(|y|) \left( \sum_{n=1}^{N-1}  a_n^{-3}  r_n^{1+\alpha} \psi(a_N) + a_N^{-1} \phi(a_N)^{-1} s_N^{-1} r_N^{1+\alpha} + \sum_{n=N+1}^{\infty}  a_n^{-2} \phi(a_N)^{-1} r_n^{1+\alpha} \right)\\
    &= C|y|\phi(|y|) \left[ (A_I) + (A_{II}) + (A_{III}) \right].
\end{align*}

Now we show that $(A_I)$, $(A_{II})$ and $(A_{III})$ all tend to $0$ as $N \to \infty$. First, note that

\begin{align*}
    (A_I) &= \psi(a_N) \sum_{n=1}^{N-1}  \psi(a_n)^{-1} n^{-1}\\
    &\leq \psi(a_N) \sum_{n=1}^{j-1}  \psi(a_n)^{-1} + \psi(a_N) j^{-1} \sum_{n=j}^{N-1}  \psi(a_n)^{-1} \\
    &\leq \psi(a_N) \psi(a_j)^{-1} + j^{-1},
\end{align*}

\noindent where the last line follows from the property that

\begin{align*}
    \sum_{n=1}^{j-1} \psi(a_n)^{-1} < \psi(a_j)^{-1}
\end{align*}

\noindent for $1\leq j < N$. Thus, by choosing $j$ sufficiently large, it follows that $(A_I) \to 0$ as $N \to \infty$. Next, it follows from the choices of $r_n$ and $s_n$ that

\begin{align*}
    (A_{II}) &= 7 a_N^{-1} \phi(a_N)^{-1} a_N^{-1} N^{1/3} a_N^2 N^{-1} \phi(a_N)\\
    &= 7 N^{-2/3}
\end{align*}

\noindent and $(A_{II}) \to 0$ as $N \to \infty$. Finally, 

\begin{align*}
    (A_{III}) &= \phi(a_N)^{-1} \sum_{n=N+1}^{\infty} n^{-1} \phi(a_n)\\
 &< \phi(a_N)^{-1} N^{-1} \sum_{n=N+1}^{\infty}  \phi(a_n) \\
    &< \phi(a_N)^{-1} N^{-1} \sum_{n=1}^{\infty} 2^{-n} \phi(a_N)\\
    & = N^{-1},
\end{align*}

\noindent  where the third line follows from the property that $\phi(a_n) < \frac{1}{2} \phi(a_{n-1})$. Thus, $(A_{III}) \to 0$ as $N \to \infty$ and by choosing $N$ sufficiently large, we obtain the following upper bound for (A):

\begin{align*}
    (A) \leq \epsilon |y| \phi(|y|).
\end{align*}

 Now we focus on getting a bound for $(B)$. Since $y \in A_N$,

\begin{align*}
    (B) &= |y|^2 \left( \sum_{n=1}^{N-1}  a_n^{-1} d_n(y)^{-2}  r_n^{1+\alpha} + a_N^{-1} d_N(y)^{-2} r_N^{1+\alpha} + \sum_{n=N+1}^{\infty}  a_n^{-1} d_n(y)^{-2}  r_n^{1+\alpha}\right)\\
    &\leq |y|^2\left( \sum_{n=1}^{N-1}  a_n^{-3}  r_n^{1+\alpha} + a_N^{-1} s_N(y)^{-2} r_N^{1+\alpha} + \sum_{n=N+1}^{\infty}  a_n^{-1} a_N^{-2}  r_n^{1+\alpha} \right).
\end{align*}

\noindent Since $|y| \leq \frac{4}{3} \phi(|y|) \psi(a_N)$,

\begin{align*}
    (B) &\leq C|y| \phi(|y|) \left( \sum_{n=1}^{N-1}  a_n^{-3}  r_n^{1+\alpha} \psi(a_N) + \phi(a_N)^{-1} s_N(y)^{-2} r_N^{1+\alpha} + \sum_{n=N+1}^{\infty}  a_n^{-1} \phi(a_N)^{-1} a_N^{-1}  r_n^{1+\alpha} \right)\\
    &= C|y| \phi(|y|)[(B_I) + (B_{II}) + (B_{III})].
\end{align*}

\noindent Now we show that $(B_I)$, $(B_{II})$ and $(B_{III})$ all tend to $0$ as $N \to \infty$. Since $(B_I) = (A_I)$ this has already been established for $(B_I)$. Next, it follows from our choices of $r_n$ and $s_n$ that 

\begin{align*}
    (B_{II}) = 49 \phi(a_N)^{-1} a_N^{-2} N^{2/3} a_N^2 N^{-1} \phi(a_N) = 49 N^{-1/3}
\end{align*}

\noindent and hence $(B_{II}) \to 0$ as $N \to \infty$. Lastly, 

\begin{align*}
    (B_{III}) &= \phi(a_N)^{-1} a_N^{-1}  \sum_{n=N+1}^{\infty} a_n n^{-1} \phi(a_n)\\
 &< \frac{1}{2} \phi(a_N)^{-1} N^{-1} \sum_{n=N+1}^{\infty} \phi(a_n)\\
    &< \frac{1}{2} \phi(a_N)^{-1} N^{-1} \sum_{n=1}^{\infty}2^{-n} \phi(a_N) \\
    &= \frac{1}{2} N^{-1},
\end{align*}

\noindent where the third line follows from the property that $\phi(a_n) < \frac{1}{2}\phi(a_n-1)$. Thus, $(B_{III}) \to 0$ as $N \to \infty$ and hence, by choosing $N$ sufficiently large, we obtain the following upper bound for (B):

\begin{align*}
    (B) \leq \epsilon |y| \phi(|y|).
\end{align*}

Thus, it follows that if $y \in E$ and $|y|$ is sufficiently small, 

\begin{align*}
    |f(y)-f(0) -f'(0)y| \leq \epsilon |y| \phi(y) ||f||_{\text{Lip}_{\alpha}},
\end{align*}

which completes the proof.

\end{proof}

Theorem \ref{main} follows from Lemma \ref{c is false} and Lemma \ref{a is true}.

\section*{Appendix}

In this section, we present an example to illustrate that the type of construction used in the proof of Theorem \ref{main} cannot be used to disprove the conjecture that $(a)$ implies $(c)$ when $t=0$. Let $U$ be a set constructed by removing closed disks $\overline{B(a_n,r_n)}$ with centers $a_n$ and radii $r_n$ from the open unit disk; that is, let $U = \Delta \setminus \bigcup \overline{B(a_n,r_n)}$. Let $\phi(r) = r^{\beta}$ where $0 < \beta < \alpha$. We will show that every choice of radii $r_n$ such that

\begin{align*}
    \sum_{n=1}^{\infty} 2^{n(1+\beta)} r_n^{1+\alpha} = \infty
\end{align*}

\bigskip
\noindent causes $U$ to be so large that $U^C$ fails to have full area density at $0$. More precisely, for $U^C$ to have full area density at $0$, it is required that

\begin{align*}
    0 = \lim_{j\to \infty} \frac{m(B(0,2^{-j})\setminus U^C)}{m(B(0,2^{-j}))} = \lim_{j\to \infty} \frac{\pi \sum_{n=j}^{\infty} r_n^2}{\pi (2^{-j})^2}.
\end{align*}

\bigskip
\noindent Since $\displaystyle \sum_{n=j}^{\infty} r_n^2> r_j^2$, it follows that for $U^C$ to have full area density at $0$, it must be the case that $2^{2j} r_j^2 \to 0$ as $j \to \infty$. Hence, it is necessary that, for $j$ sufficiently large, $r_j < 2^{-j}$. But this implies that

\begin{align*}
    \sum_{n=1}^{\infty} 2^{n(1+\beta)} r_n^{1+\alpha} < \infty,
\end{align*}

\bigskip
\noindent and thus $U^C$ cannot have full area density at $0$.

\bibliographystyle{abbrv}
\bibliography{bibliography}

\end{document}